\documentclass[11pt,a4paper]{amsart}
\usepackage[text={16cm,24.5cm}]{geometry}

\usepackage{amsfonts, amsmath, amssymb, amsthm}
\usepackage[mathscr]{eucal}
\usepackage{hyperref}

\usepackage{graphicx,color}

\usepackage{booktabs}
\usepackage{bbm}
\usepackage[latin1]{inputenc}
\usepackage{psfrag}
\usepackage{overpic}

\newcommand{\NN}{{\mathbb{N}}}
\newcommand{\QQ}{{\mathbb{Q}}}
\newcommand{\ZZ}{{\mathbb{Z}}}
\newcommand{\RR}{{\mathbb{R}}}
\newcommand{\cH}{{\mathcal{H}}}
\newcommand{\cT}{{\mathcal{T}}}
\newcommand{\cZ}{{\mathcal{Z}}}
\newcommand{\vones}{(1,1,\dots,1)}
\newcommand{\TA}{{\mathbb{TA}}}
\newcommand{\trop}{\mathrm{trop}}
\newcommand{\Assoc}{\mathrm{Assoc}}
\newcommand{\SetOf}[2]{\{#1\vphantom{#2} \mid \vphantom{#1}#2\}}
\newcommand{\SetOfbig}[2]{\bigl\{#1\vphantom{#2} \mid \vphantom{#1}#2\bigr\}}

\newcommand{\conv}{\operatorname{conv}}
\newcommand{\tconv}{\operatorname{tconv}}
\newcommand{\val}{\operatorname{val}}
\newcommand{\tdet}{\operatorname{tdet}}
\newcommand{\pos}{\operatorname{pos}}
\newcommand{\interior}{\operatorname{int}}
\newcommand{\core}{\operatorname{core}}
\newcommand{\lk}{\operatorname{lk}}
\newcommand{\st}{\operatorname{st}}
\newcommand{\type}{\operatorname{type}}
\newcommand{\Sym}{\operatorname{Sym}}

\newcommand\valid[1]{\overline{X_{#1}}\cap (v_{#1}+\bar{S}_{#1})}
\newtheorem{theorem}{Theorem}
\newtheorem{proposition}[theorem]{Proposition}

\newtheorem{lemma}[theorem]{Lemma}

\newtheorem{example}[theorem]{Example}

\newtheorem{remark}[theorem]{Remark}
\begingroup
\theoremstyle{plain}
\newtheorem{proposition_plain}[theorem]{Proposition}
\endgroup

\graphicspath{ {Pictures/} }

\begin{document}

\title[Tropical and Ordinary Convexity]{Tropical and Ordinary Convexity Combined}

\author[Joswig and Kulas]{Michael Joswig \and Katja Kulas}
\thanks{Michael Joswig is partially supported by DFG as a member of Research Unit ``Polyhedral Surfaces''.}
\address{Fachbereich Mathematik, TU Darmstadt, 64289 Darmstadt, Germany}
\email{\{joswig,kulas\}@mathematik.tu-darmstadt.de}

\date{\today}

\begin{abstract}
  A polytrope is a tropical polytope which at the same time is convex
  in the ordinary sense.  A $d$-dimensional polytrope turns out to be
  a tropical simplex, that is, it is the tropical convex hull of $d+1$
  points.  This statement is equivalent to the known fact that the
  Segre product of two full polynomial rings (over some field $K$) has
  the Gorenstein property if and only if the factors are generated by
  the same number of indeterminates.  The combinatorial types of
  polytropes up to dimension three are classified.
\end{abstract}

\maketitle

\section{Introduction}

In \cite{DevelinSturmfels04} Develin and Sturmfels defined tropical
polytopes, and they showed that tropical polytopes, or rather
configurations of $n$ tropical points in the tropical affine space
$\TA^d$, are equivalent to regular subdivisions of the product of
simplices $\Delta_{n-1}\times\Delta_d$.  It is important that there is
a natural way to identify $\TA^d$ with $\RR^d$; this way it is
possible to carry geometric concepts from $\RR^d$ to $\TA^d$.  A key
result \cite{DevelinSturmfels04}, Theorem~15, says that each tropical
polytope comes naturally decomposed into ordinary polytopes which are
also convex in the tropical sense.  These objects are the topic of
this paper, and we call them \emph{polytropes}.

Each polytrope $P$ is a tropical simplex, that is, it is the tropical
convex hull of $d+1$ points, where $d$ is the dimension of~$P$.  It
turns out that this statement is equivalent to the known fact from
Commutative Algebra that the Segre product of two full polynomial
rings (over some field $K$) has the Gorenstein property if and only if
the factors are generated by the same number of indeterminates.

Polytropes are not new.  Postnikov and Stanley studied
\emph{deformations} of the Coxeter hyperplane arrangement of type
A$_d$, that is, arrangements of affine hyperplanes in $\RR^d$ with
normal vectors $e_i-e_j$ for $i\ne j$ \cite{PostnikovStanley00}; here
$e_1,e_2,\dots,e_d$ are the standard basis vectors of $\RR^d$.  Their
bounded cells are precisely the polytropes.  In a paper by Lam and
Postnikov \cite{LamPostnikov05} the same objects are called the
\emph{alcoved polytopes} of type A.~More recently, polytropes occurred
as the bounded intersections of apartments in Bruhat--Tits buildings
of type $\tilde{\mathrm{A}}_d$, see Keel and
Tevelev~\cite{KeelTevelev07} or Joswig, Sturmfels, and
Yu~\cite{JoswigSturmfelsYu07}, as the \emph{inversion domains} of
Alessandrini~\cite{Alessandrini07}, and as the \emph{max-plus definite
  closures} of Sergeev~\cite{Sergeev07}.  An additional motivation to
study polytropes comes from the fact that each tropical polytope $P$
has a canonical decomposition into polytropes.

The paper is structured as follows.  We begin with a short section
gathering the relevant facts about tropical polytopes.  Then we prove
our main result, and this section also contains more information about
the interplay between the tropical and the ordinary convexity of a
polytrope.  The subsequent section lists specific examples, among
which are the associahedra and order polytopes.  One application of
our main result is that it allows for a fairly efficient (compared
with other more obvious approaches) enumeration of all combinatorial
types of polytropes.  We sketch the procedure, and we report on our
complete classification of the $3$-dimensional polytropes.  The final
section deals with the relationship to Commutative Algebra mentioned
above.

We are indebted to Tim R\"omer and Bernd Sturmfels for valuable
discussions on the subject.  A first set of examples of polytropes was
computed by Edward D.~Kim, and we are grateful that he shared his
results with us.  Moreover, we would like to thank Serge{\u\i} Sergeev
for his comments on a preprint version of this paper.

\section{Tropical convexity}

This section is meant to collect basic facts about tropical convexity
and to fix the notation.

Defining \emph{tropical addition} $x\oplus y:=\min(x,y)$ and
\emph{tropical multiplication} $x\odot y:=x+y$ yields the
\emph{tropical semi-ring} $(\RR,\oplus,\odot)$.  Component-wise
tropical addition and \emph{tropical scalar multiplication}
\begin{equation*}
  \lambda \odot (\xi_0,\dots,\xi_d)  :=  (\lambda \odot
  \xi_1,\dots,\lambda \odot \xi_d)  = (\lambda+\xi_0,\dots,\lambda+\xi_d)
\end{equation*}
equips $\RR^{d+1}$ with a semi-module structure.  For
$x,y\in\RR^{d+1}$ we let
\begin{equation*} [x,y]_\trop := \SetOf{(\lambda \odot x) \oplus (\mu
    \odot y)}{\lambda,\mu \in \RR}
\end{equation*}
be the \emph{tropical line segment} between $x$ and $y$.  A subset of
$\RR^{d+1}$ is \emph{tropically convex} if it contains the tropical
line segment between any two of its points.  A direct computation
shows that if $S\subset\RR^{d+1}$ is tropically convex then $S$ is
closed under tropical scalar multiplication.  This leads to the
definition of the \emph{tropical affine space} as the quotient
semi-module
\begin{equation*}
  \TA^d  :=  \RR^{d+1} / (\RR\odot(0,\dots,0))   .
\end{equation*}

Note that $\TA^d$ was called ``tropical projective space'' in
\cite{DevelinSturmfels04}, \cite{Joswig05}, \cite{DevelinYu06}, and
\cite{JoswigSturmfelsYu07}.  Tropical convexity gives rise to the hull
operator $\tconv$.  A \emph{tropical polytope} is the tropical convex
hull of finitely many points in $\TA^d$.

Like an ordinary polytope each tropical polytope $P$ has a unique set
of generators which is minimal with respect to inclusion; these are
the \emph{tropical vertices} of $P$.

There are several natural ways to choose a representative coordinate
vector for a point in $\TA^d$.  For instance, in the coset
$x+(\RR\odot(0,\dots,0))$ there is a unique vector $c(x)\in\RR^{d+1}$
with non-negative coordinates such that at least one of them is zero;
we refer to $c(x)$ as the \emph{canonical coordinates} of $x\in\TA^d$.
Moreover, in the same coset there is also a unique vector
$(\xi_0,\dots,\xi_d)$ such that $\xi_0=0$.  Hence the map
\begin{equation}\label{eq:c_0}
  c_0   :   \TA^d \to \RR^d   ,  (\xi_0,\dots,\xi_d) \mapsto (\xi_1-\xi_0,\dots,\xi_d-\xi_0)
\end{equation}
is a bijection.  Often we will identify $\TA^d$ with $\RR^d$ via this
map. This is also sound from the topological point of view: The
maximum norm on $\RR^{d+1}$ induces a metric on $\TA^d$ and, in this
way, a natural topology; the map $c_0$ is a homeomorphism.

The \emph{tropical determinant} $\tdet M$ of a matrix $M=(\mu_{ij})\in
\RR^{(d+1)\times(d+1)}$ is given as
\begin{equation}\label{eq:tdet}
  \tdet M  :=  \bigoplus_{\sigma\in\Sym_{d+1}} \mu_{0,\sigma(0)}+\dots+\mu_{d,\sigma(d)}   ,
\end{equation}
where $\Sym_{d+1}$ denotes the symmetric group of degree $d+1$ acting
on the set $\{0,1,\dots,d\}$. In the literature this is also called
the ``min-plus permanent'' of $M$. The matrix
$M\in\RR^{(d+1)\times(d+1)}$ is \emph{tropically singular} if the
minimum in \eqref{eq:tdet} is attained at least twice.

The \emph{tropical hyperplane} $\cH_a$ defined by the \emph{tropical
  linear form} $a=(\alpha_0,\dots,\alpha_d)\in\RR^{d+1}$ is the set of
points $(\xi_0,\dots,\xi_d)\in\TA^d$ such that the minimum
\begin{equation*}
  (\alpha_0 \odot \xi_0) \oplus \dots \oplus (\alpha_d \odot \xi_d)
\end{equation*}
is attained at least twice.  The complement of a tropical hyperplane
in $\TA^d$ has exactly $d+1$ connected components, each of which is an
\emph{open sector}. A \emph{closed sector} is the topological closure
of an open sector.  The set
\begin{equation*}
  S_k  :=  \SetOfbig{(\xi_0,\dots,\xi_d)\in\TA^d}{\xi_k=0 \text{ and }
    \xi_i>0 \text{ for } i\ne k}   ,
\end{equation*}
for $0\le k \le d$, is the \emph{$k$-th open sector} of the tropical
hyperplane $\cZ$ in $\TA^d$ defined by the zero tropical linear form.
Its closure is
\begin{equation*}
  \bar S_k  :=  \SetOfbig{(\xi_0,\dots,\xi_d)\in\TA^d}{\xi_k=0 \text{ and
    } \xi_i\ge 0 \text{ for } i\ne k}   .
\end{equation*}
We also use the notation $\bar S_I:=\bigcup\SetOf{\bar S_i}{i\in I}$
for any set $I\subset\{0,\dots,d\}$.

If $a=(\alpha_0,\dots,\alpha_d)$ is an arbitrary tropical linear form
then the translates $-a+S_k$ for $0\le k\le d$ are the open sectors of
the tropical hyperplane $\cH_a$.  The point $-a$ is the unique point
contained in all closed sectors of $\cH_a$, and it is called the
\emph{apex} of $\cH_a$.  For each $I\subset\{0,1,\dots,d\}$ with $1\le
\# I \le d$ the set $-a+\bar S_I$ is the \emph{closed tropical
  halfspace} of $\cH_a$ of type $I$.  The tropical polytopes in
$\TA^d$ are exactly the bounded intersections of finitely many closed
tropical halfspaces; see \cite{Joswig05} and \cite{GaubertKatz09}.

The points $v_1,\dots,v_n\in\TA^d$ are in \emph{tropically general
  position} if the $n\times(d+1)$-matrix whose $i$-th row is $v_i$ has
no $k\times k$-submatrix which is tropically singular, for $2\le k\le
\min(n,d+1)$.

Note that the integral translates of the hyperplanes $x_i=x_j$ induce
a triangulation of $\RR^d=c_0(\TA^d)$; this is called the \emph{alcove
  triangulation} $\TA_\Delta^d$ of $\TA^d$ by Lam and
Postnikov~\cite{LamPostnikov05}.

A \emph{tropical $d$-simplex} in $\TA^d$ is the tropical convex hull
of $d+1$ points in $\TA^d$ which are not contained in the boundary of
a tropical halfspace; see Figure~\ref{fig:trop-simplex}.  It must be
stressed that the vertices of a tropical simplex are not necessarily
in tropically general position.  For example, see the first tropical
triangle in Figure~\ref{fig:trop-simplex}.

\begin{figure}[htb]
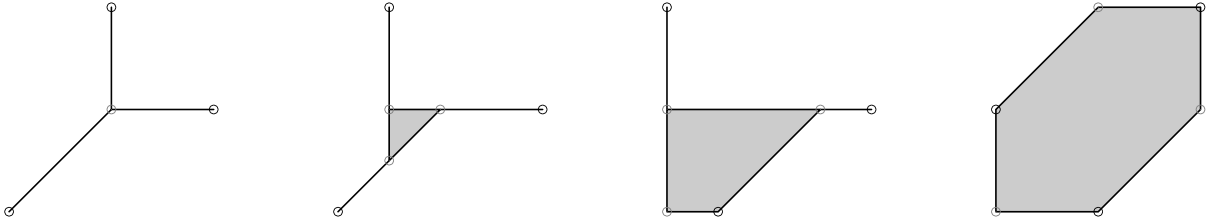


  \includegraphics[width=.18\textwidth]{trop-simplex.0}\hfill
  \includegraphics[width=.18\textwidth]{trop-simplex.1}\hfill
  \includegraphics[width=.18\textwidth]{trop-simplex.2}\hfill
  \includegraphics[width=.18\textwidth]{trop-simplex.3}

  \caption{Four tropical simplices in $\TA^2$ (with their tropical
    vertices drawn black).  The vertices of the first one are not in
    tropically general position.}
  \label{fig:trop-simplex}
\end{figure}

Let $V:=(v_1,\dots,v_n)$ be a sequence of points in $\TA^d$.  The
\emph{type} of $x\in\TA^d$ with respect to $V$ is the ordered
$(d+1)$-tuple $\type_V(x):=(T_0,\dots,T_d)$ where
\begin{equation*}
  T_k  :=  \SetOf{i\in\{1,\dots,n\}}{v_i\in x+\bar S_k}   .
\end{equation*}
For a given type $\cT$ with respect to $V$ the set
\begin{equation*}
  X_V(\cT)  :=  \SetOfbig{x\in\TA^d}{\type_V(x)=\cT}
\end{equation*}
is the \emph{cell} of type $\cT$ with respect to $V$.  With respect
to inclusion the types with respect to $V$ form a partially ordered
set.

\begin{figure}[htb]
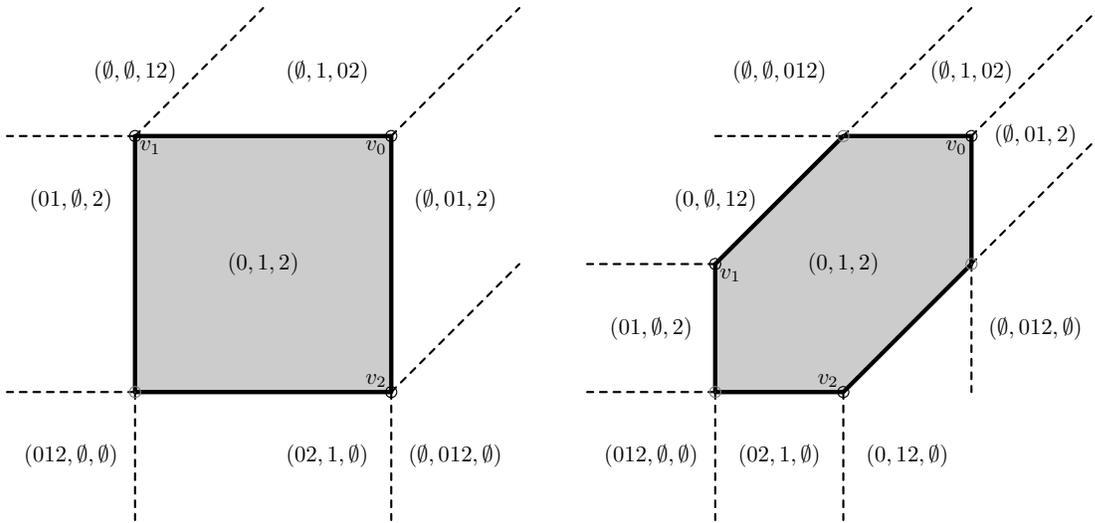

  \strut
  \hfill\includegraphics[height=.43\textwidth]{types.0} \hfill
  \includegraphics[height=.43\textwidth]{types.3}\hfill
  \strut

  \caption{Types and maximal cells with respect to two different
    triplets of points in $\TA^2$.}
  \label{fig:quad-hexagon-types}
\end{figure}

The symmetric group $\Sym(\{0,\dots,d\})$ acts on $\TA^d$ by permuting
the coordinates.  This operation fixes $\TA^d$, and it preserves
inclusion of sets as well as ordinary and tropical convexity.
Further, tropical hyperplanes are mapped to tropical hyperplanes, and
tropical halfspaces are mapped to tropical halfspaces.  This action
gives rise to a natural equivalence of point configurations: Two
sequences $V=(v_1,\dots,v_n)$ and $W=(w_1,\dots,w_n)$ of points in
$\TA^d$ are \emph{tropically equivalent} if there is a pair of
permutations
\begin{equation*}
  (\sigma,\tau)   \in   \Sym(\{1,\dots,n\}) \times \Sym(\{0,\dots,d\})
\end{equation*}
such that the map
\begin{equation*}
  (T_0,\dots,T_d)   \mapsto   (U_{\tau(0)},\dots,U_{\tau(d)}),
\end{equation*}
where $U_i=\sigma(T_i)$, is a poset isomorphism from the types with
respect to $V$ to the types with respect to~$W$.  Occasionally, it
will also be convenient to start the numbering of the vertices with
zero rather than one.

\begin{remark}\label{rem:lattice}
  A decisive difference to ordinary point configurations in $\RR^d$ is
  that each tropical point configuration in $\TA^d$ has a tropically
  equivalent realization with integral vertices.
\end{remark}

Two tropical polytopes are said to be \emph{tropically equivalent} if
their tropical vertices are tropically equivalent as point
configurations.  Figure~\ref{fig:quad-hexagon-types} shows two
tropical triangles which are not tropically equivalent.  Develin and
Sturmfels \cite{DevelinSturmfels04}, Theorem~1, showed that the
tropical equivalence classes of $n$ points in $\TA^d$ are dual to the
regular subdivisions of the product of simplices
$\Delta_{n-1}\times\Delta_d$.  By \cite{DevelinSturmfels04},
Proposition~24, the regular subdivision of
$\Delta_{n-1}\times\Delta_d$ dual to the point configuration $V$ is a
triangulation if and only if $V$ is in tropically general position.
Figure~6 in \cite{DevelinSturmfels04} shows all 35 tropical
equivalence classes of quadruples of points in $\TA^2$.

We will now discuss a link between tropical and ordinary convexity via
Puiseux series; for the general picture see Speyer and
Sturmfels~\cite{SpeyerSturmfels04}, Theorem~2.1, and
Markwig~\cite{Markwig07}.  Let $K=\RR((t^{1/\infty}))$ be the field of
Puiseux series with real coefficients.  It is known that $K$ is real
closed which is why its first order theory coincides with the first
order theory of the reals; see Salzmann et al.~\cite{SalzmannEtAl07},
\S64.24.  In particular, there are ordinary convex polytopes in $K^d$,
and they behave much like ordinary polytopes in $\RR^d$.  An element
of $K$ can be written as $f=\sum_{i\ge N} a_i t^{i/n}$ for some
$N\in\ZZ$ and $n\in\NN$.  In particular, if $f\ne 0$ there is a
minimal $d\in\ZZ$ such that $a_d\ne 0$.  We call $d/n$ the
\emph{(lower) degree} of $f$ and denote it by $\val f$. By setting
$\val(0)=\infty$ the map $\val:K\to\QQ\cup\{\infty\}$ is a
valuation. This gives rise to
\begin{equation*}
  \val   :   K^d   \to   (\QQ\cup\{\infty\})^d   , 
  (f_1,\dots,f_d)   \mapsto   (\val f_1,\dots,\val f_d)   .
\end{equation*}
In \cite{DevelinYu06}, Proposition~2.1, it is shown that each tropical
polytope $P$ in $\TA^d$ (identified with $\RR^d$ via the map $c_0$
from~\eqref{eq:c_0}) with rational coordinates arises as the image of
an ordinary convex polytope in $K^d$ under the map $\val$.  Each
element of the fiber will be called a \emph{Puiseux lifting} of
$P$. In the same way tropical hyperplanes are images of ordinary
hyperplanes, tropical halfspaces are images of ordinary halfspaces,
and tropical point configurations can be lifted to $K^d$.

\begin{lemma}\label{lem:d+1-facets}
  Let $P\subset\TA^d$ be the intersection of $d+1$ tropical
  halfspaces.  Then one of the following holds:
  \begin{enumerate}
  \item\label{it:d+1-facets:a} $P$ is unbounded or
  \item\label{it:d+1-facets:b} $P$ is contained in a tropical
    hyperplane or
  \item\label{it:d+1-facets:c} $P$ is a tropical simplex.
  \end{enumerate}
\end{lemma}

Clearly, the properties \eqref{it:d+1-facets:a} and
\eqref{it:d+1-facets:c} are mutually exclusive, while the two other
combinations can occur together.

\begin{proof}
  Let us first assume that the apices of the tropical hyperplanes have
  rational coordinates.  Then the properties above are inherited from
  ordinary convexity via a Puiseux lifting from $\TA^d$ to~$K^d$.  If
  the coordinates of the apices are irrational then we can perturb the
  situation to rational (or even integer) coordinates in view of
  Remark~\ref{rem:lattice}.
\end{proof}

\section{Polytropes}

A subset of $\TA^d$ is \emph{convex in the ordinary sense} if its image in
$\RR^d$ under the map $c_0$ as in \eqref{eq:c_0} is convex.  A
\emph{polytrope} is a tropical polytope which is also convex in the ordinary
sense.  In order to avoid confusion, we call the vertices of a polytrope, seen
as an ordinary polytope, its \emph{pseudo-vertices}.  A $d$-dimensional
polytrope $P$, or \emph{$d$-polytrope} for short, has exactly one bounded cell
of dimension~$d$ with respect to its vertices: its interior.  This is called
the \emph{basic cell}, and its type (with respect to the tropical vertices of
$P$) is the \emph{basic type} of~$P$.

\begin{remark}
  An ordinary polytope which additionally is tropically convex is not
  necessarily a polytrope: For example, the ordinary triangle
  $\conv\{(0,0),(2,1),(0,1)\}$ is tropically convex.  However, this is not a
  tropical polytope since it is not the tropical convex hull of any finite
  subset of $\TA^2$.
\end{remark}

In order to investigate polytropes any further it is useful to look at
the root systems of type A$_d$; see Bourbaki~\cite{Bourbaki02} for the
complete picture.  The relationship to polytropes is the following.
The root system of type A$_d$ consists of the $d(d+1)$ vectors
$e_i-e_j$ in $\RR^{d+1}$ with $0\le i,j\le d$ and $i\ne j$.  Call an
ordinary convex polyhedron whose (outer) facet normals (scaled to
Euclidean length $\sqrt{2}$) form a subset of those roots an
\emph{ordinary A$_d$-polyhedron}.  Since it contains the ray
$\RR(1,1,\dots,1)$ an ordinary A$_d$-polyhedron is always unbounded.
Its intersection with the coordinate hyperplane $x_0=0$ has facet
normals
\begin{equation}\label{eq:facet-normals}
  \pm e_i \text{ and }  e_i-e_j \quad \text{for }  1\le i,j\le d  \text{ and } i\ne j   .
\end{equation}
Moreover, the tropical hyperplanes in $\TA^d$ are formed from pieces of
ordinary affine hyperplanes with such normal vectors.  It then follows from
\cite[Lemma~10]{DevelinSturmfels04} that the polytropes are precisely the
intersections of ordinary A$_d$-polyhedra with the coordinate hyperplane
$x_0=0$.  The latter were called \emph{alcoved polytopes of type A} by Lam and
Postnikov~\cite{LamPostnikov05}.

\begin{example}
  The classification of polytropes is the topic of Section~\ref{sec:enum}
  below.  Here we list the result in the planar case $d=2$.  Up to tropical
  equivalence there are exactly five types of $2$-polytropes.  Considered as
  ordinary polygons, they have three, four, five, and six pseudo-vertices,
  respectively; see Figure~\ref{fig:polytropes2d}.
\end{example}

\begin{figure}[htb]
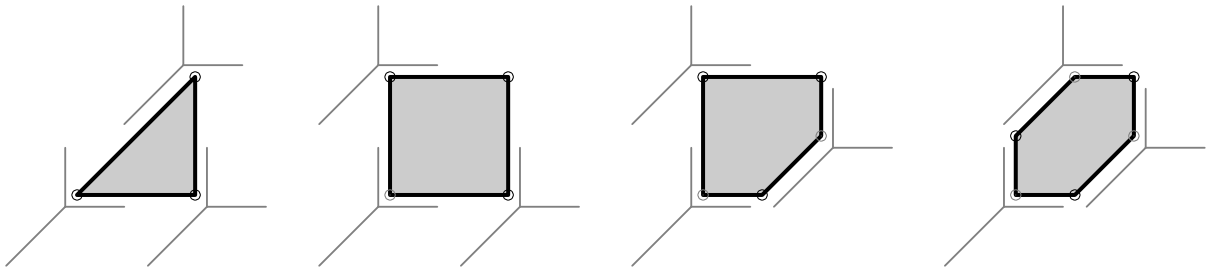

  \includegraphics[width=.22\textwidth]{polytropes2d.3}\hfill
  \includegraphics[width=.22\textwidth]{polytropes2d.4}\hfill
  \includegraphics[width=.22\textwidth]{polytropes2d.5}\hfill
  \includegraphics[width=.22\textwidth]{polytropes2d.6}

  \caption{Four types of polytropes in $\TA^2$.  The tropical vertices
    are black, and the pseudo-vertices are grey.  The sketches of
    tropical hyperplanes indicate the facet defining tropical
    halfspaces.}
  \label{fig:polytropes2d}
\end{figure}

\begin{proposition}\label{prop:ordinary-facet-upper-bound}
  Each $d$-polytrope has at most $d(d+1)$ ordinary facets, and this
  bound is sharp.
\end{proposition}

\begin{proof}
  The upper bound is clear since $d(d+1)$ is the number of roots of
  type A$_d$.  That this bound is sharp follows from the construction
  below.
\end{proof}

The maximum number of ordinary facets is attained, for instance, by
the \emph{$d$-pyrope}
\begin{equation}\label{eq:big-standard}
  \Pi_d  :=  \tconv(-e_0,-e_1,\dots,-e_d)   .
\end{equation}
The name is inspired by the fact that \emph{pyrope} is a mineral whose
structure as a pure crystal can take the form of a rhombic
dodecahedron, and the latter is combinatorially equivalent to $\Pi_3$
as an ordinary polytope; see Figure~\ref{fig:zono} for a picture.  The
chemical sum formula of pyrope is Mg$_3$Al$_2$(SiO$_4$)$_3$; see
Anthony et al.~\cite{pyrope} for the mineralogy facts.  In general,
$\Pi_d$ is a cubical zonotope with $d+1$ zones, which can be written
as $\conv ([0,1]^d\cup[-1,0]^d)$.  The number of its pseudo-vertices
equals $2^d-2$.

\begin{figure}[htb]
  \centering
  \includegraphics[width=.5\textwidth,clip=true]{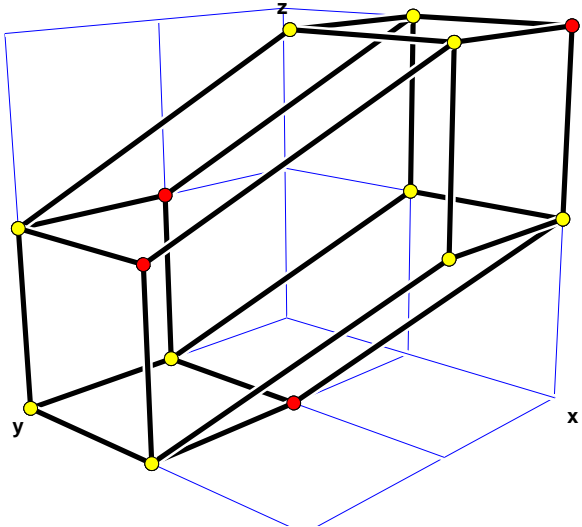}
  \caption{The pyrope $\Pi_3$ is a polytrope which, as an ordinary
    polytope, is a rhombic dodecahedron.}
  \label{fig:zono}
\end{figure}

To obtain the exact upper bound for the number of pseudo-vertices of a
polytrope is less trivial.  A class of polytropes attaining the upper
bound on the number of pseudo-vertices will be constructed in the next
section.

\begin{proposition_plain}
  (Gelfand, Graev, and Postnikov \cite{GelfandGraevPostnikov97},
  Theorem~2.3(2); Develin and Sturm\-fels~\cite{DevelinSturmfels04},
  Proposition~19).\label{prop:pseudo-vertex-upper-bound}
  Each $d$-polytrope has at most $\tbinom{2d}{d}$ pseudo-vertices, and
  this bound is sharp.
\end{proposition_plain}

The corresponding questions concerning the lower bounds are trivial:
The \emph{small tropical $d$-simplex}
\begin{equation}\label{eq:small-standard}
  \tconv(0,e_1,e_1+e_2,\dots,e_1+e_2+\dots+e_d)
\end{equation}
is also an ordinary simplex, hence the obvious lower bound of $d+1$
for the number of ordinary facets as well as for the number of
pseudo-vertices is actually attained.

Following \cite[Proposition 18]{DevelinSturmfels04} we are now going to
describe how to obtain the tropical vertices of a polytrope $P$ from an
ordinary inequality description.  As in \eqref{eq:facet-normals} we assume
that $P$ is the set of points in $\TA^d$, identified with $\RR^d$, satisfying
the inequalities
\begin{equation}\label{eq:P:ineq}
  x_i-x_j \le c_{ij} \quad \text{ for all } (i,j)\in J  ,
\end{equation}
where $J$ is a subset of $\SetOf{(i,j)}{i,j\in\{0,\dots,d\}, i\ne j}$
and $x_0=0$.  Since $P$ is bounded, the set of vectors
\begin{equation*}
  \SetOf{e_i-e_j}{(i,j)\in J} \cup \{\pm(1,1,\dots,1)\}
\end{equation*}
positively spans $\RR^{d+1}$.  The last two vectors do not correspond
to facet normals, but they make up for the fact that an ordinary
A$_d$-polyhedron is always unbounded.  We will construct a sequence
$V=(v_0,\dots,v_d)$ of $d+1$ points which will turn out to be the
tropical vertices of $P$.  The computation will be organized in a way
such that the basic type of $P$ with respect to $V$ is
$(0,1,\dots,d)$.  Each tropical vertex satisfies at least $d$ of the
inequalities \eqref{eq:P:ineq} with equality.  This is immediate from
the fact that each tropical vertex of $P$ is also a pseudo-vertex,
that is, an ordinary vertex of $P$.

First we may assume that each inequality in the description
\eqref{eq:P:ineq} is tight, that is, that the corresponding ordinary
affine hyperplane supports~$P$.  Second we may assume that each root
vector of type A$_d$ actually gives one inequality in the description
\eqref{eq:P:ineq}.  If this assumption is not initially given it can
explicitly be established as follows.  Let $(i,k)\not\in J$, that is,
the corresponding inequality is initially not given.  If
$(i,j_1),(j_1,j_2),\dots,(j_{m-1},j_m)$, $(j_m,k)$ are in $J$ then the
equation
\begin{equation*}
  x_i-x_k  =  x_i-x_{j_1}+x_{j_1}-x_{j_2}+x_{j_2}-\dots+
  x_{j_{m-1}}-x_{j_m}+x_{j_m}-x_k
\end{equation*}
leads to the definition
\begin{equation*}
  c_{ik}  :=  c_{ij_1}+c_{j_1j_2}+\dots+c_{j_{m-1}j_m}+c_{j_mk}   ,
\end{equation*}
and $x_i-x_k \le c_{ik}$ is a new tight inequality.  Iterating this
procedure gives all the inequalities desired since
$\SetOf{e_i-e_j}{(i,j)\in J} \cup \{\pm(1,1,\dots,1)\}$ positively
spans $\RR^{d+1}$.  Now the coordinates $(v_{i0},\dots,v_{id})$ of the
point $v_i$ are uniquely determined by setting $v_{i0}=0$ and $d(d+1)$
more equations.  An equivalent but more symmetric requirement is
\begin{equation}\label{eq:P:v}
  v_{ii} = 0 \quad \text{and} \quad v_{ik}  =  c_{ki} \quad \text{ for } i\ne k   .
\end{equation}
This computation is equivalent to the Floyd--Warshall algorithm for
computing all shortest paths in a directed graph \cite{CLRS01}.
Lemma~10 of \cite{DevelinSturmfels04} proves the following.

\begin{theorem}\label{thm:main}
  The $d+1$ points in the sequence $V=(v_0,v_1,\dots,v_d)$ defined in
  \eqref{eq:P:v} are the tropical vertices of the $d$-poly\-trope~$P$.
  In particular, each polytrope is a tropical simplex.
\end{theorem}

\begin{example}
  We wish to give an example of how to compute the tropical vertices
  of a polytrope from an ordinary inequality description.  Let $P$ be
  the $2$-polytrope described by the inequalities $x_1\le 2$, $-x_1\le
  0$, $x_2\le 2$, $-x_2\le 0$, $x_1-x_2\le 1$; this looks like the
  third tropical triangle in Figure~\ref{fig:polytropes2d}, which is
  an ordinary pentagon.  All inequalities are tight.  The unique
  initially missing inequality corresponds to $e_2-e_1$.  We compute
  $x_2-x_1=x_2-x_0+x_0-x_1$ and
  \begin{equation*}
    c_{21}  =  c_{20} + c_{01}  =  2+0  =  2   .
  \end{equation*}
  Hence the missing inequality is $x_2-x_1\le 2$.  From this we infer
  that $v_0=(0,c_{10},c_{20})=(0,2,2)$,
  $v_1=(c_{01},0,c_{21})=(0,0,2)$, and
  $v_2=(c_{02},c_{12},0)=(0,1,0)$.  With respect to these generators
  the type of the basic cell reads $(0,1,2)$.
\end{example}

The tropical halfspaces containing a tropical polytope $P$ are
partially ordered by inclusion.  A tropical halfspace which is
minimal with respect to this partial ordering and which has the
additional property that its apex is a pseudo-vertex, is called
\emph{facet defining} for $P$.  It is known that $P$ is the
intersection of its (finitely many) facet defining tropical
halfspaces.  Notice that the proof in \cite[Theorem~3.6]{Joswig05}
uses \cite[Proposition~3.3]{Joswig05} which is wrong.  A corrected
statement is due to Gaubert and Katz
\cite[Proposition~1]{GaubertKatz09}, and this suffices to prove
\cite[Theorem~3.6]{Joswig05}; see also
\cite[Theorem~2]{GaubertKatz09}.  As in Lemma~\ref{lem:d+1-facets} one
can use Puiseux liftings to show that if $P$ is a full-dimensional
polytrope $P$ in $\TA^d$ it has exactly $d+1$ facet defining tropical
halfspaces.  Here we give a direct and constructive proof.

For an arbitrary sequence $V=(v_1,\dots,v_n)$ of points in $\TA^d$ and
$k\in\{0,\dots,d\}$ let
\begin{equation*}
  c_k(V) := ((-v_{1,k})\odot v_1) \oplus ((-v_{2,k})\odot v_2)\oplus \dots \oplus
  ((-v_{n,k})\odot v_n)
\end{equation*}
be the $k$-th \emph{corner} of $P=\tconv(V)$.  By construction each
corner belongs to the tropical convex hull $P$.  It is also obvious
that the \emph{cornered tropical halfspace} $c_k+\bar S_k$ contains
$P$.  Notice that the corners of $P$ do not depend on the choice of
the set of generators $V$.  We say that $P=\tconv(V)\subset\TA^d$ is
\emph{full-dimensional} if its dimension as an ordinary polytopal
complex in $\RR^d$ equals $d$.  Here we do not assume that $P$ is a
polytrope.

\begin{proposition}
  Suppose that $P$ is a full-dimensional tropical polytope.  Then the
  $d+1$ cornered tropical halfspaces are facet defining tropical
  halfspaces of~$P$.
\end{proposition}

\begin{proof}
  The $k$-th corner $c_k=(c_{k0},\dots,c_{kd})$ is contained in the
  $d$ ordinary affine hyperplanes $x_k-x_l=c_{kk}-c_{kl}$ for all
  $l\in\{0,\dots,d\}\setminus\{k\}$.  The corresponding $d$ normal
  vectors $e_k-e_l$ are skew to the vector $\vones$, and hence they
  linearly span the quotient $\RR^d=\TA^d$.  Therefore, the
  intersection of these hyperplanes is a point.  This implies that
  $c_k$ is a vertex of the max-tropical hyperplane arrangement induced
  by $V$, which means that $c_k$ is a pseudo-vertex.

  Suppose that $c_k+\bar S_k$ is not minimal. Then there must be some
  other tropical halfspace $w + \bar S_K$ contained in $c_k + \bar
  S_k$ which still contains $P$.  Without loss of generality we can
  assume that $w + \bar S_K$ is minimal and thus $w\in P$.  Since
  $c_k+\bar S_k$ consists of a single closed sector it follows that
  $K=\{0\}$.  Moreover, since $c_k$ is contained in $P$, we have
  $c_k-w\in\bar S_k$.  However, we also have $w\in c_k+\bar S_k$ since
  $c_k+\bar S_k$ contains all points of $P$.  We conclude that
  $w=c_k$, and this proves that each cornered halfspace is facet
  defining.
\end{proof}

\begin{proposition}
  If $P$ is a polytrope then the cornered tropical halfspaces are
  the only facet defining tropical halfspaces of~$P$.
\end{proposition}

\begin{proof}
  Let $(v_0,v_1,\dots,v_d)$ be the tropical vertices of the
  $d$-polytrope $P$.  Up to a transformation we can assume that the
  basic type of $P$ is $(0,1,\dots,d)$.  Moreover, we assume that the
  coordinates are chosen such that $v_{ii}=0$ holds for all
  $i\in\{0,\dots,d\}$.  We have to show that there are no other facet
  defining tropical halfspaces for $P$.  By construction the
  \emph{cornered hull}
  \begin{equation}\label{eq:cornered_hull}
    (c_0+\bar S_0) \cap (c_1+\bar S_1) \cap \dots \cap (c_d+\bar S_d)
  \end{equation}
  of $P$ is the convex polyhedron subject to the $d(d+1)$ ordinary
  inequalities $x_i-x_k\ge c_{ik}$ for all $i\ne k$.   Equivalently,
  we have
  \begin{equation*}
    x_i-x_k\le -c_{ki} \quad \text{ for } i\ne k
  \end{equation*}
  Since the cornered hull is bounded it is a polytrope.  We can apply
  the procedure \eqref{eq:P:v} to get at the tropical vertices of the
  cornered hull.  These are exactly the points $v_0,\dots,v_d$, and
  the claim follows.
\end{proof}

\begin{remark}
  Since $c_k+\bar S_k$ contains all points in $V$ we have that
  $T_k=\{1,\dots,n\}$ where $T=(T_0,\dots,T_d)$ is the type of $c_k$
  with respect to $V$.  The $k$-th corner is the unique pseudo-vertex
  of $P$ with this property.
\end{remark}

\begin{remark}
  If $V\in\RR^{(d+1)\times(d+1)}$ is a matrix (with zero diagonal)
  whose rows correspond to the tropical vertices of a polytrope $P$
  then the rows of its negative transpose $-V^t$ yield the corners.
  It follows that the corners are the tropical vertices of $P$, seen
  as a max-tropical polytope.  The map $V\mapsto-V^t$ is an instance
  of the duality of tropical polytopes discussed in
  \cite[Theorem~23]{DevelinSturmfels04}.
\end{remark}

\goodbreak

\section{Constructions and examples}

\subsection{Associahedra}\label{sec:assoc}

Studying expansive motions Rote, Santos, and Streinu arrived at
interesting new realizations of the associahedra
\cite{RoteSantosStreinu03}, \S5.3.  They consider the polyhedron in
$\RR^n$ which is defined by
\begin{align}\label{eq:associahedron}
  x_j-x_i &  \ge  (i-j)^2 \qquad \text{for $1\le i<j\le n$} \notag \\
  x_1 &  =  0 \\
  x_n & = (n-1)^2 . \notag
\end{align}
This turns out to be an ordinary polytope which is combinatorially
equivalent to the $(n-2)$-dimensional associahedron, which is a
secondary polytope of a convex $(n+1)$-gon.  The $\tbinom{n}{2}-1$
inequalities $x_j-x_i\ge(i-j)^2$ for $(i,j)\ne(1,n)$ are all facet
defining.

If we project the polytope defined in \eqref{eq:associahedron}
orthogonally onto the subspace of $\RR^n$ spanned by the standard
basis vectors $e_2,e_3,\dots,e_{n-1}$ we obtain a full-dimensional
realization $\Assoc_{n-2}\subset\RR^{n-2}$ which is tropically convex
(via the identification from \eqref{eq:c_0}).  That is, $\Assoc_{n-2}$
is a polytrope.

We can apply the procedure from \eqref{eq:P:v} to determine the
tropical vertices of $\Assoc_{n-2}$.  Each tropical vertex will be
described by listing the $n-1$ ordinary facets containing it.  If the
ordinary facet $x_j-x_i = (i-j)^2$ from~\eqref{eq:associahedron} is
denoted as $(i,j)$ then the $j$-th tropical vertex of $\Assoc_{n-2}$,
where $1\le j\le n-1$, is the intersection of the facets
$(1,j),(2,j),\dots,(j-1,j),(j+1,n),(j+2,n),\dots,(n-1,n)$. For
example, the tropical vertices of $\Assoc_{3}$ are
\begin{equation*}
  \begin{array}{lcl}
    (2,5),(3,5),(4,5) & = & (7,12,15)   ,\\
    (1,2),(3,5),(4,5) & = & (1,12,15)   ,\\
    (1,3),(2,3),(4,5) & = & (3,4,15),  \quad  \text{and}\\
    (1,4),(2,4),(3,4) & = & (5,8,9)   .
  \end{array}
\end{equation*}
On the right hand side are the coordinates in $\RR^3$.  The polytrope
$\Assoc_{2}$ is an ordinary pentagon like in
Figure~\ref{fig:polytropes2d} (third).

\subsection{Polytropes with many pseudo-vertices}

We want to construct a class of polytropes which attain the upper
bound on the number of pseudo-vertices from
Proposition~\ref{prop:pseudo-vertex-upper-bound}.  This construction
is an explicit instance of what arises from the proof of
\cite{DevelinSturmfels04}, Proposition~19.  The following lemma says
that we can perturb the vertices of the pyrope $\Pi_d$
from~\eqref{eq:big-standard} quite a bit, and we still have a
polytrope.

\begin{lemma}\label{lem:perturb}
  For an arbitrary matrix
  $E=(\varepsilon_{ik})_{i,k}\in[0,\frac{1}{2})^{(d+1)\times(d+1)}$
  the tropical polytope
  \begin{equation*}
    \Pi_d^E  := \tconv(-e_0+\varepsilon_{0,\cdot},-e_1+\varepsilon_{1,\cdot},\dots,-e_d+\varepsilon_{d,\cdot})
  \end{equation*}
  is a polytrope.
\end{lemma}
	
\begin{proof}
  A direct computation shows that the generators
  $-e_0+\varepsilon_{0,\cdot},\dots,-e_d+\varepsilon_{d,\cdot}$, in
  fact, are the tropical vertices of $\Pi_d^\varepsilon$.  For the
  rest of the proof we fix this particular vertex ordering.

  Observe that the type of the origin is $(0,1,\dots,d)$.  Now we
  compute the type $(T_0,T_1$, $\dots,T_d)$ of the vertex
  $-e_i+\varepsilon_{i,\cdot}$.  We claim that $T_k=\{i,k\}$ if $i\ne
  k$ and $T_i=\{i\}$.  Indeed, for $i\ne j$ we have
  $-e_j+\varepsilon_{j,\cdot}\in-e_i+\varepsilon_{i,\cdot}+\bar S_k$
  if and only if
  $e_i-e_j+\varepsilon_{j,\cdot}-\varepsilon_{i,\cdot}\in \bar S_k$ if
  and only if $j=k$ since $0\le
  \varepsilon_{ik},\varepsilon_{jk}<\frac{1}{2}$.

  From this we learn that each vertex is contained in the closure of
  the cell of type $(0,1,\dots,d)$, and hence there is only one
  bounded cell.
\end{proof}

For a random matrix $E$ Lemma~\ref{lem:perturb} would yield a
polytrope with the maximal number of vertices (almost surely).  The
following is a deterministic solution.

\begin{example}
  For any fixed positive $\varepsilon$ with $\varepsilon<\frac{1}{2}$
  let
  \begin{equation*}
    E  =  
    \begin{pmatrix}
      0 & \varepsilon & \varepsilon^2 & \dots & \varepsilon^{d-1} & \varepsilon^d \\
      \varepsilon^d & 0 & \varepsilon & \varepsilon^2 & \dots & \varepsilon^{d-1} \\
      \varepsilon^{d-1} & \varepsilon^d & 0 & \varepsilon & \dots & \varepsilon^{d-2} \\
      \vdots & \ddots & \ddots & \ddots & \ddots & \vdots \\
      \varepsilon^2 & \dots & \varepsilon^{d-1} & \varepsilon^d & 0 & \varepsilon \\
      \varepsilon & \varepsilon^2 & \dots & \varepsilon^{d-1} &
      \varepsilon^d & 0
    \end{pmatrix} .
  \end{equation*}
  Then the perturbed pyrope $\Pi_d^E$ is a polytrope with
  $\tbinom{2d}{d}$ pseudo-vertices, which is the upper bound from
  Proposition~\ref{prop:pseudo-vertex-upper-bound}.
\end{example}

There is only one tropical type of $2$-polytrope attaining the upper bound six
on the number of pseudo-vertices, shown in Figure~\ref{fig:polytropes2d}
(fourth).  Already in dimension $3$, however, there are five distinct types of
polytropes with $20$ vertices, which are also pairwise not combinatorially
equivalent as ordinary polytopes.  All of them are simple and share the same
$f$-vector $(20,30,12)$.  For each of the five types we give a
$4\times4$-matrix such that the tropical convex hull of the rows gives the
corresponding polytrope; these are also shown in Figure~\ref{fig:max3d}:
\begin{equation*}\scriptscriptstyle
  \begin{pmatrix}
    0 & 0 & 2 & 2 \\
    4 & 0 & 4 & 3 \\
    4 & 3 & 0 & 4 \\
    6 & 4 & 4 & 0
  \end{pmatrix}
  , \hfill
  \begin{pmatrix}
    0 & 2 & 2 & 2 \\
    4 & 0 & 4 & 2 \\
    4 & 3 & 0 & 4 \\
    6 & 6 & 4 & 0
  \end{pmatrix}
  , \hfill
  \begin{pmatrix}
    0 & 10 & 11 & 14 \\
    14 & 0 & 10 & 11 \\
    11 & 14 & 0 & 10 \\
    10 & 11 & 14 & 0
  \end{pmatrix}
  , \hfill
  \begin{pmatrix}
    0 & 1 & 2 & 2 \\
    8 & 0 & 8 & 7 \\
    10 & 6 & 0 & 8 \\
    6 & 5 & 4 & 0
  \end{pmatrix}
  , \hfill
  \begin{pmatrix}
    0 & 6 & 6 & 2 \\
    6 & 0 & 2 & 3 \\
    11 & 10 & 0 & 10 \\
    8 & 8 & 9 & 0
  \end{pmatrix}
\end{equation*}
\begin{figure}[htb]
  \hspace*{-2mm}
  \includegraphics[width=.19\textwidth]{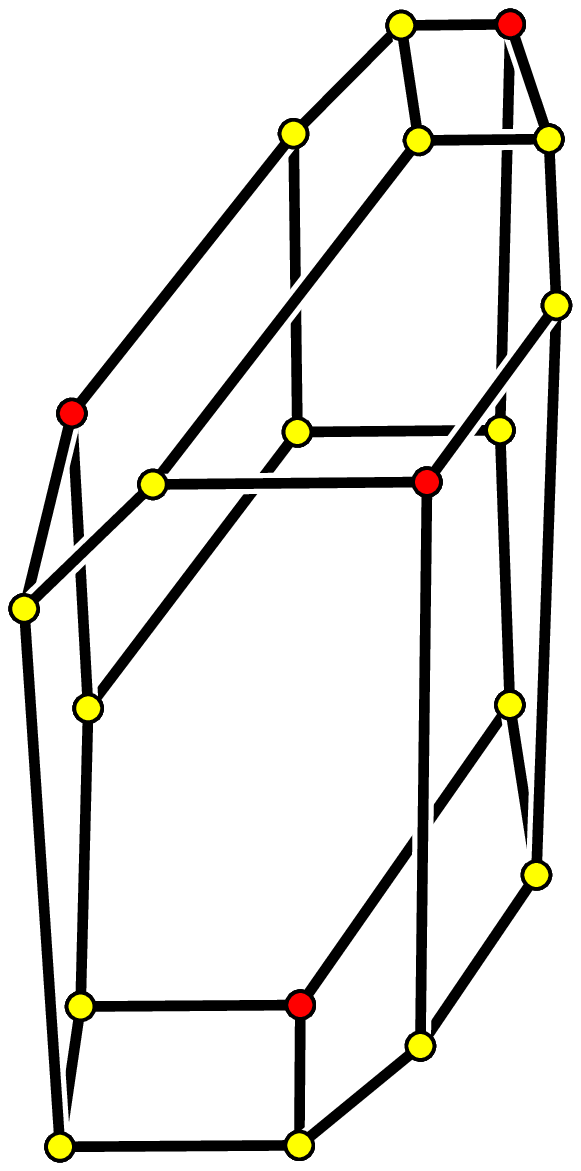}
  \hspace*{-3mm}
  \includegraphics[width=.19\textwidth]{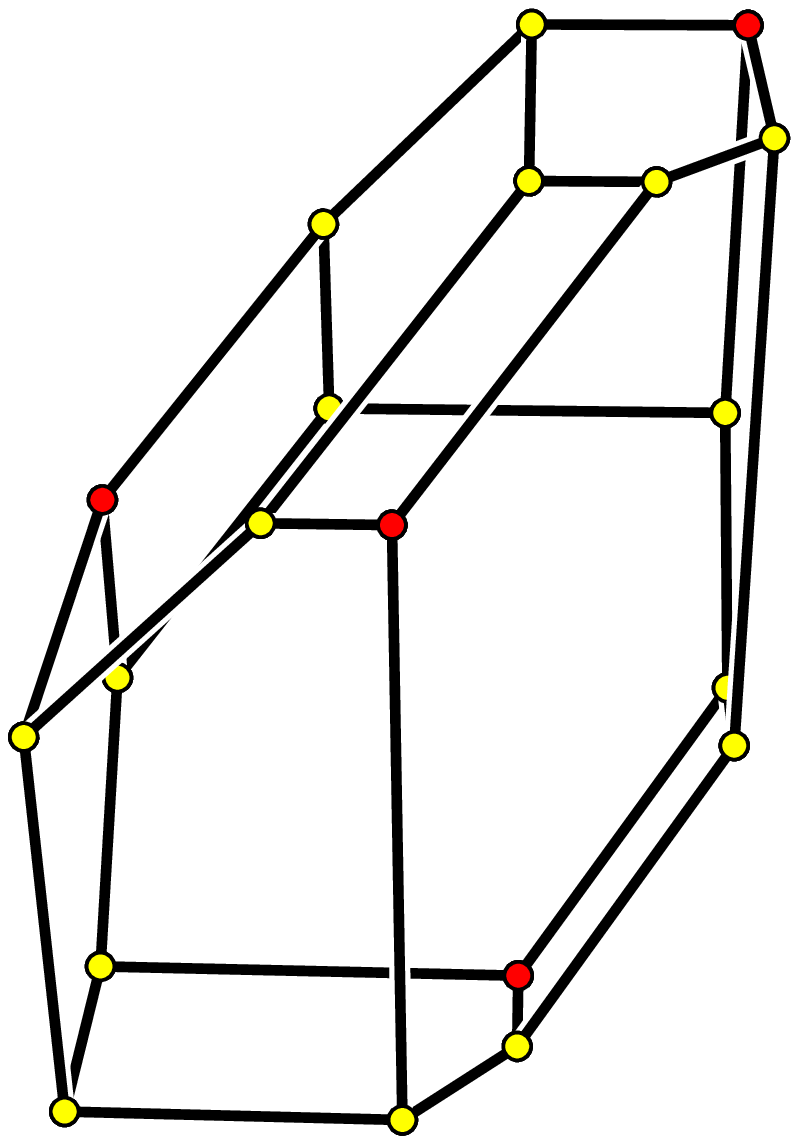}
  \hspace*{-0.5mm}
  \includegraphics[width=.19\textwidth]{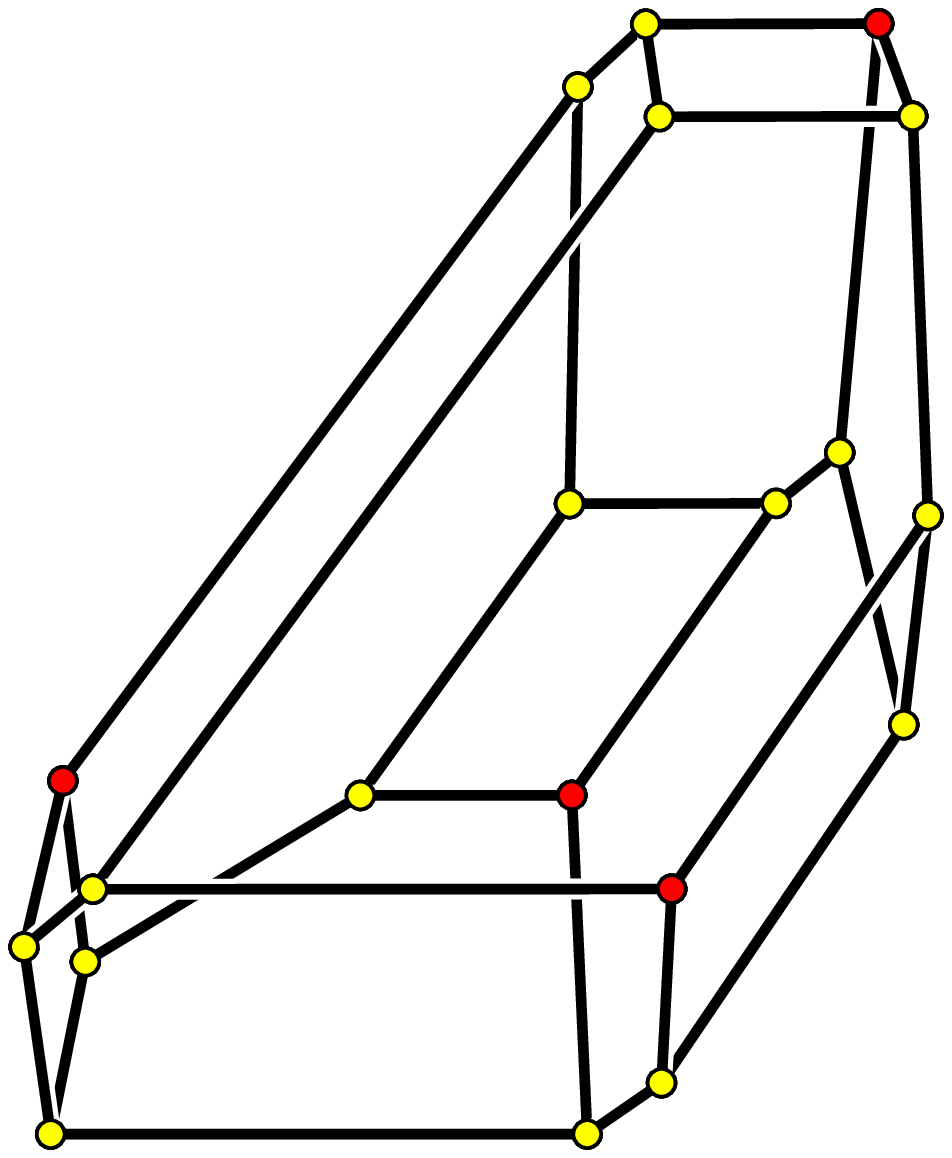}
  \hspace*{1.5mm}
  \includegraphics[width=.19\textwidth]{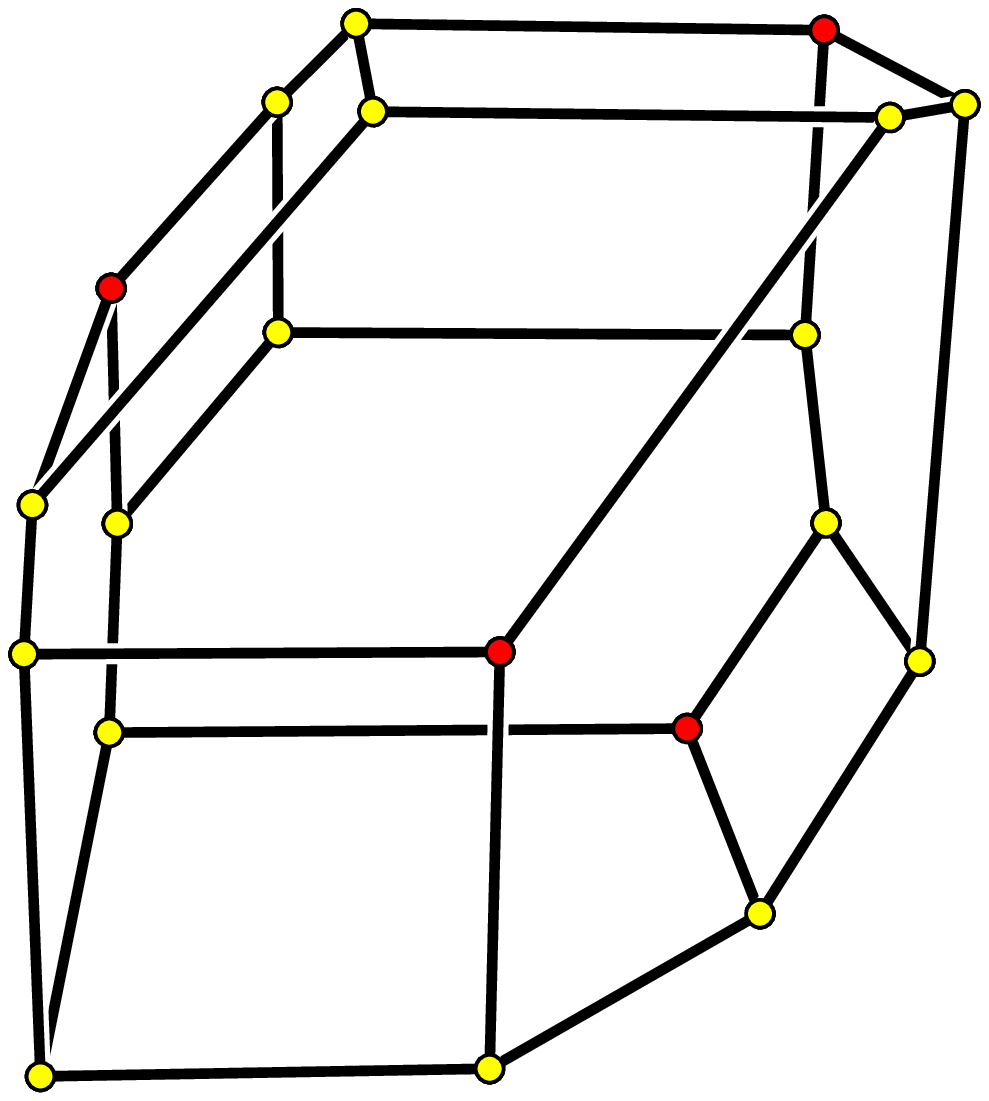}
  \hspace*{-.5mm}
  \includegraphics[width=.19\textwidth]{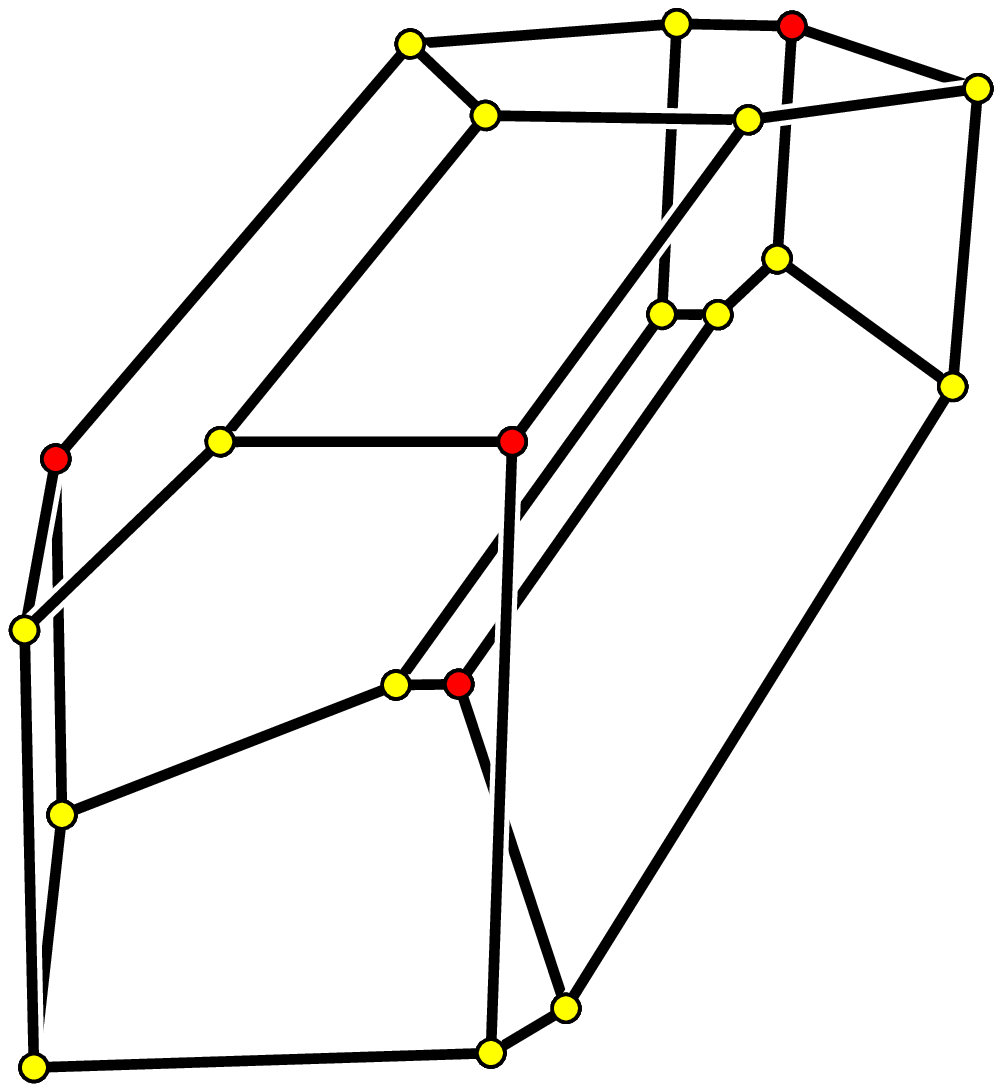}
  \caption{The five tropical types of $3$-polytropes with 20
    vertices.}
  \label{fig:max3d}
\end{figure}

\section{Enumerating all polytropes}
\label{sec:enum}

We want to explain how to enumerate all polytropes in $\TA^d$ for
fixed $d$.  Since their number of pseudo-vertices (and ordinary
facets) is bounded by
Propositions~\ref{prop:ordinary-facet-upper-bound}
and~\ref{prop:pseudo-vertex-upper-bound} it is clear that there are
only finitely many distinct tropical types.  Of course, in principal,
it is possible to enumerate all regular subdivisions of
$\Delta_d\times\Delta_d$ and to sort out those which are dual to a
polytrope; see \cite{Rambau:TOPCOM-ICMS:2002,PfeifleRambau03}.  But
this does not seem to be practically feasible even for $d=3$ due to
the sheer size of the secondary polytope of $\Delta_3\times\Delta_3$.
However, there is a more efficient approach which will be the subject
of the discussion now.  The efficiency will be underlined by being
able to achieve a complete classification of the tropical types of
$3$-polytropes; we think that even the $4$-dimensional case is within
reach.

In view of Remark~\ref{rem:lattice} we can restrict our attention to
enumerating \emph{lattice polytropes}, that is, polytropes whose
pseudo-vertices have integral coordinates.  Since the alcove
triangulation $\TA_\Delta^d$ induces a triangulation on any lattice
polytrope, and since the small tropical simplex from
\eqref{eq:small-standard} is a maximal face of $\TA_\Delta^d$ it
suffices to enumerate integral polytropes which contain the small
tropical simplex.  This means that we can obtain each (tropical type
of) polytrope by successively adding generators outside the small
tropical simplex.

Throughout the following we look at a $d$-polytrope
$P=\tconv(v_0,\dots,v_d)\subset\TA^d$, and we assume that the basic
type of $P$ is $(0,1,\dots,d)$, which is equivalent to requiring that
$\tdet(v_0,\dots,v_d)=v_{00}+\dots+v_{dd}$.  Our type computations
will be with respect to this ordering of the vertices of~$P$.

Let $(T^i_0,T^i_1,\dots,T^i_d)=\type_V(v_i)$.  Since $v_i$ is a
tropical vertex of $P$ we have $T^i_i=\{i\}$.  Moreover, as the basic
type is $(0,1,\dots,d)$ we have $k\in T^i_k$ for all $i,k$.  In this
situation the tropical halfspace $v_i+\bar{S}_i$ intersects $P$ only
in the vertex $v_i$.  The set $v_i+\bar{S}_i$ is always contained in
the normal cone of $v_i$ seen as a vertex of the ordinary polytope
$P$.

In Figure~\ref{pic:Regionsoutside} the light regions form the tropical
halfspaces $v_i+\bar{S}_i$.  For a new point $x$ the tropical polytope
$P(x):=\tconv\{v_0,\dots,v_d,x\}$ will be convex in the ordinary sense
or not, depending on the type of $x$.

\begin{figure}[htb]
  \centering
  \psfrag{S0}{$v_0+\bar{S}_0$} \psfrag{S1}{$v_1+\bar{S}_1$}
  \psfrag{S2}{$v_2+\bar{S}_2$} \psfrag{v0}{$v_0$} \psfrag{v1}{$v_1$}
  \psfrag{v2}{$v_2$}
  \includegraphics[scale=0.45]{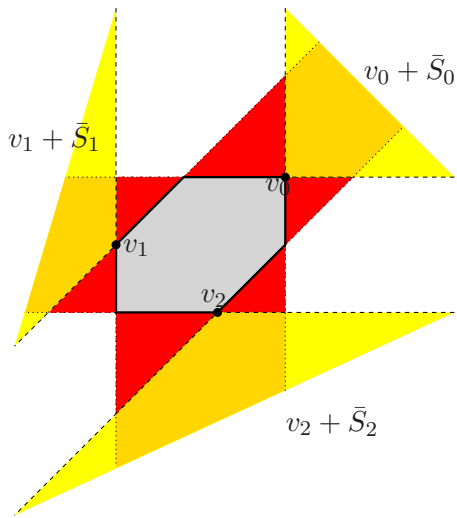}
  \caption{The light/yellow regions are the tropical halfspaces
    $v_i+\bar{S}_i$, the dark/red ones are the cells $X_{i,j}$ of type
    $(T_{i,j,0},T_{i,j,1},T_{i,j,2})$ as in \eqref{eq:T_ijk}.  See
    also Figure~\ref{fig:quad-hexagon-types} (right).}
  \label{pic:Regionsoutside}
\end{figure}

\begin{proposition}\label{prop:regions}
  Consider the union of cells $X_i=X_{i,0}\cup\dots\cup X_{i,i-1}\cup
  X_{i,i+1}\cup\dots\cup X_{i,d}$, where $X_{i,j}$ is the cell of type
  $(T_{i,j,0},T_{i,j,1},\dots,T_{i,j,d})$, and where
  \begin{equation}\label{eq:T_ijk}
    T_{i,j,k}  =     \begin{cases}
      \emptyset & \text{if $k=i$}   ,\\
      T^i_k\cup\{i\} & \text{if $k=j$}   ,\\
      T^i_k & \text{otherwise}   .
    \end{cases}
  \end{equation}
  Then the tropical polytope $P(x)=\tconv(v_0,\dots,v_d,x)$ is convex
  in the ordinary sense if and only if
  \begin{equation*}
    x   \in   \bigcup_{i=0}^d (\valid{i})   ,
  \end{equation*}
  where $\overline{X_i}$ is the topological closure of $X_i$.
  Moreover, in this case we have $P(x)\supseteq P$, so $v_i$ is
  redundant in $P(x)$.
\end{proposition}

In order to give it a concise name we call the set $\valid{i}$ the
\emph{$i$-th valid region} with respect to $V$.

\begin{proof}
  First let us assume that $x\in X_{i,j}\cap(v_i+\bar{S}_i)$.  By
  symmetry we can assume that $i=0$.  This is to say that
  \begin{equation*}
    \type_V(x)  =   (\emptyset,T^0_1,T^0_2,\dots,T^0_{j-1},T^0_{j}\cup\{0\},T^0_{j+1},\dots,T^0_{d})
    .
  \end{equation*}
  As $x$ is contained in $v_0+\bar{S}_0$ it follows that
  $P(x)=\tconv(x,v_1,\dots,v_d)$.

  We have to show that $P(x)$ is convex in the ordinary sense.  To
  this end we fix a point $z$ in the basic cell of $P$.  Then
  $\type_V(z)=(0,1,\dots,d)$.  Since $x\in v_0+\bar{S}_0$, and since
  the other vertices remain the same we conclude that the type of $z$
  with respect to $V(x):=(x,v_1,\dots,v_d)$ is also $(0,1,\dots,d)$.
  Now we compute the type $(U_0,U_1,\dots,U_d)$ of $x$ with respect to
  $V(x)$.  Clearly, $0\in U_0$.  If we can show that for all
  $k\in\{1,\dots,d\}$ we have $k\in U_k$ then it follows that $x$ is
  in the boundary of the cell of type $(0,1,\dots,d)$, and hence
  $P(x)$ is a polytrope.

  So we assume that there is some $k\in\{1,\dots,d\}$ with $k\not\in
  U_k$. Since $v_0$ is the only point that is now missing in the
  sequence of generators we know that $U_k\supseteq
  T^0_k\setminus\{0\}$.  Actually, since $x$ is a tropical vertex of
  $P(x)$, we even have $U_k\supseteq T^0_k$.  By construction $j\in
  T^0_j\subseteq U_j$, and also $k\in T^0_k\subseteq U_k$ for
  $k\in\{1,2,\dots,d\}\setminus\{j\}$ because $v_0$ is contained in
  the boundary of the basic cell of~$P$.

  It remains to prove the converse: We have to show that if
  $x\notin\bigcup_{i=0}^d (X_i\cap v_i+\bar{S}_i)$ then $P(x)$ is not
  convex in the ordinary sense.  We distinguish two cases.  If
  $x\notin P\cup\bigcup_{i=0}^d (v_i+\bar{S}_i)$ then none of the
  generators of $P(x)$ is redundant, and, due to
  Theorem~\ref{thm:main}, the tropical polytope $P(x)$ cannot be a
  polytrope.

  Finally, let
  $x\in\bigcup_{i=0}^d(v_i+\bar{S}_i)\setminus\bigcup_{i=0}^d X_i$.
  Again, by symmetry we can assume that $x\in v_0+\bar{S}_0$.  As
  above $P(x)\supseteq P$.  Then if $(U_0,U_1,\dots,U_d):=\type_V(x)$
  there is some $j\ne 0$ such that $U_j=\emptyset$.  It can be shown
  that the point $y:=\frac{1}{2}(x+v_j)$ lies outside $P(x)$, whence
  $P(x)$ is not convex in the ordinary sense.
\end{proof}

With the aid of Proposition~\ref{prop:regions} we can enumerate all
tropical equivalence types of polytropes.  Consider a polytrope
$P=\tconv(v_0,\dots,v_d)$ in $\TA^{d}$ and its valid regions
$\valid{i}$.  Simultaneously choosing $d+1$ points $v_i'\in\valid{i}$
with $i\in\{0,\dots,d\}$ the tropical convex hull
$\tconv(v_0',\dots,v_d')$ is a polytrope because the valid regions
with respect to the old points $v_0,\dots,v_d$ are contained in the
valid regions of the new points $v_0',\dots,v_d'$.  Moreover, if the
types of $(v_0',\dots,v_d')$ are the same as $(v_0'',\dots,v_d'')$
then the resulting polytropes $\tconv(v_0',\dots,v_d')$ and
$\tconv(v_0'',\dots,v_d'')$ are tropically equivalent.

For our initial points $v_0,\dots,v_d$ we take the (tropical) vertices
of the small tropical $d$-simplex scaled by $d$.  The advantage of
this scaling is that each cell in the valid regions contains (at
least) one integral point.  The tropical convex hulls of $d+1$ such
points, one from each valid region, yield all the tropical types of
polytropes in $\TA^d$.  In order to enumerate all tropical equivalence
classes it suffices to consider one (integral) point per cell within
each valid region.

For an efficient procedure it is essential to take symmetries into
account.

We implemented this enumeration scheme in
\texttt{polymake}~\cite{polymake}, and the result of the computation
for $d=3$ is given in Table~\ref{tab:classification}.  Here $t_3(m)$
is the number of tropical equivalence classes of $3$-polytropes with
exactly $m$ pseudo-vertices, and $o_3(m)$ is the corresponding number
of combinatorial types of ordinary polytopes.  We necessarily have
$o_d(m)\le t_d(m)$ for all choices of $m$ and $d$.  From
Proposition~\ref{prop:pseudo-vertex-upper-bound} we know that the
maximum number of pseudo-vertices equals $\tbinom{6}{3}=20$.

\begin{table}[hbt]
  \caption{Tropical and ordinary equivalence classes of polytropes in $\TA^3$.}
  \renewcommand{\arraystretch}{0.9}
  \vspace*{.2cm}
  \centering
  \begin{tabular*}{.6\linewidth}{@{\extracolsep{\fill}}ccc@{}}
    \toprule
    $m$ & $t_3(m)$ & $o_3(m)$\\
    \midrule
    4&1&1\\
    5&1&1\\
    6&4&2\\
    7&3&3\\
    8&20&6\\
    9&14&6\\
    10&39&13\\
    11&43&14\\
    12&68&27\\
    13&54&22\\
    14&74&31\\
    15&53&30\\
    16&43&31\\
    17&21&20\\
    18&17&17\\
    19&8&8\\
    20&5&5\\
    \bottomrule
  \end{tabular*}
  \label{tab:classification}
\end{table}

The total numbers are $\sum_{m=4}^{20} t_3(m)=468$ and
$\sum_{m=4}^{20} o_3(m)=237$.  To locate some special examples in
Table~\ref{tab:classification} that occurred above: The (up to
tropical equivalence) unique $3$-polytrope with $4$ pseudo-vertices is
the small tropical tetrahedron. The $3$-pyrope $\Pi_3$ from
Figure~\ref{fig:zono} has $2^4-2=14$ pseudo-vertices; the
associahedron $\Assoc_3$ from Section~\ref{sec:assoc} also has $14$
pseudo-vertices, but it is not even combinatorially equivalent to
$\Pi_3$.  The five classes of $3$-polytropes with $20$ pseudo-vertices
are shown in Figure~\ref{fig:max3d}.

\section{Gorenstein simplicial complexes and Gorenstein polytopes}

From Theorem~\ref{thm:main} and \cite{DevelinSturmfels04},
Proposition~24, we know that $d$-polytropes in $\TA^d$, identified
with the tropical point configuration of their tropical vertices, are
dual to triangulations of the product of simplices
$\Delta_d\times\Delta_d$.  The purpose of this section is to view
these triangulations as abstract simplicial complexes and to
interpret them in terms of Commutative Algebra.  In particular, this
way we will obtain an alternate proof of Theorem~\ref{thm:main}.

A standard construction of new simplicial complexes from old ones is
iterative coning.  For the following it is crucial to determine if a
given simplicial complex has been obtained in such a way.  Let
$\Delta$ be an arbitrary simplicial complex on a finite vertex set
$V=\{x_1,\dots,x_n\}$. As usual we let
\begin{align*}
  \st_\Delta\sigma  &:=  \SetOf{\tau\in\Delta}{\sigma\cup\tau\in\Delta}   ,\\
  \lk_\Delta\sigma  &:=  \SetOf{\tau\in\st_\Delta\sigma}{\sigma\cap\tau=\emptyset}   ,\\
  \core V  &:=  \SetOf{v\in V}{\st_\Delta v\ne \Delta}, \quad  \text{and}\\
  \core\Delta &:= \Delta_{\core V} ,
\end{align*}
where $\Delta_W$ is the subcomplex of $\Delta$ induced on the vertices
$W\subseteq V$.  By construction $\Delta_{V \setminus \core V}$ is a
simplex, and $\Delta$ is the join of $\core V$ with $\Delta_{V
  \setminus \core V}$.

For a field $K$ let $K[\Delta]$ be the \emph{Stanley--Reisner} ring of
$\Delta$, that is, 
\begin{equation*}
K[\Delta]:=K[x_1,\dots,x_n]/I_\Delta 
\end{equation*}
where
$I_\Delta$ is the ideal generated by the monomials whose exponent
vectors correspond to characteristic functions of the (minimal)
non-faces of $\Delta$. A direct computation shows that
\begin{equation*}
  K[\Delta]  =  K[\core\Delta][x \mid x\in V\setminus\core V]   ,
\end{equation*}
that is, $K[\Delta]$ is the full polynomial ring with coefficients
$K[\core\Delta]$ and indeterminates indexed by $V\setminus\core V$.  A
simplicial complex $\Delta$ is called \emph{Gorenstein} if $K[\Delta]$
is a Gorenstein ring.  Further, a positively $\ZZ^d$-graded ring $R$
is \emph{Gorenstein} if it is Cohen--Macaulay, and the Matlis dual of
the top local cohomology is isomorphic to a $\ZZ^d$-graded translate
of $R$; see \cite{Stanley96}, Section~I.12.  More useful for our
purposes is the following characterization.

\begin{theorem}[Stanley \cite{Stanley96}, Theorem~5.1]\label{thm:Stanley:5.1} A simplicial
  complex $\Delta$ is Gorenstein (over a field $K$) if and only if
  for all $\sigma\in\core\Delta$ we have
  \begin{equation*}
    \tilde H_i(\Gamma;K)  =   \begin{cases}
      K & \text{if $i=\dim\Gamma$}\\
      0 & \text{otherwise,}
    \end{cases}
  \end{equation*}
  where $\Gamma=\lk_{\core\Delta}\sigma$.
\end{theorem}

Here $\tilde H_i(\Gamma;K)$ is the $i$-th reduced (simplicial)
homology of $\Gamma$ with coefficients in $K$.  The characterization
requires $\Gamma$ to have the same homology (with coefficients in $K$)
as the sphere of dimension $\dim\Gamma$.  The \emph{tight span} of a
triangulation is its dual cell complex.

\begin{proposition}\label{prop:Gorenstein}
  A regular triangulation of an ordinary polytope is Gorenstein
  (over an arbitrary field $K$) if and only if its tight span has
  a unique maximal cell.
\end{proposition}

\begin{proof}
  Let $\Delta$ be a regular triangulation of an ordinary polytope $P$.
  First suppose that the tight span $\Delta^*$ consists of a single
  maximal cell. Hence there is a simplex $\sigma\in\Delta$ in the
  interior of $P$ which is contained in each maximal simplex
  of~$\Delta$. The vertices of $\sigma$ are precisely the cone points
  of $\Delta$, and $\Delta$ is the join of $\sigma$ with
  $\lk_\Delta\sigma=\core\Delta$.  The link of an interior face in a
  triangulated manifold (with or without boundary) is a simplicial
  sphere.  By Theorem~\ref{thm:Stanley:5.1} it follows that $\Delta$
  is Gorenstein.

  Conversely, let $\Delta$ be Gorenstein. Again, by
  Theorem~\ref{thm:Stanley:5.1}, we know that
  \begin{equation*}
    \Delta=\Delta_{V\setminus\core V}*\core\Delta, 
  \end{equation*}
  where $V$ is the vertex set of $P$ (and $\Delta$), $V\setminus\core
  V\ne\emptyset$, and $\core\Delta$ is an orientable pseudomanifold.
  Then $\Delta_{V\setminus\core V}$ is an interior simplex contained
  in all maximal simplices of $\Delta$, and hence
  $\Delta_{V\setminus\core V}$ corresponds to the unique maximal cell
  of $\Delta^*$.
\end{proof}

Let $P$ be an ordinary lattice $d$-polytope embedded into the affine
hyperplane $\RR^d\times\{1\}$ of $\RR^{d+1}$.  Then $M(P)=\pos
P\cap\ZZ^{d+1}$ is the set of lattice points in the positive cone
spanned by $P$ in $\RR^{d+1}$. Now $P$ is a \emph{Gorenstein polytope}
if there exists $u \in \interior M(P)$ such that
\begin{equation}\label{eq:Gorenstein}
  \interior M(P) = u + M(P)   ,
\end{equation}
see Bruns and Herzog \cite{BrunsHerzog93}, Chapter~6.  Here $\interior
M(P)=(\pos P\setminus\partial(\pos P))\cap\ZZ^{d+1}$ denotes the set
of interior lattice points of $M(P)$.  Gorenstein polytopes and their
Gorenstein triangulations are related as follows; see also Conca,
Ho{\c{s}}ten, and Thomas~\cite{ConcaHostenThomas06}.

\begin{theorem}[Bruns and R\"omer \cite{BrunsRoemer07}, Corollary~8]\label{thm:BrunsRoemer}
  Let $P$ be an ordinary lattice $d$-poly\-tope with some regular and
  unimodular triangulation using all the lattice points in $P$. Then
  $P$ is a Gorenstein polytope if and only if it has some regular
  triangulation which is Gorenstein.
\end{theorem}

Now there is the following well-known result; for far generalizations
see Goto and Watanabe~\cite{GotoWatanabe78}, Theorem~4.4.7.  For the
sake of completeness we give a simple proof.

\begin{theorem}\label{thm:main:var}
  The product of simplices $\Delta_m\times\Delta_n$ is a Gorenstein
  polytope if and only if $m=n$.
\end{theorem}

\begin{proof}
  The simplex $\Delta_n=\conv(0,e_1,\dots,e_n)$ is a Gorenstein
  polytope by the criterion \eqref{eq:Gorenstein}, since the scaled
  simplex $k\Delta_n$ contains precisely one interior lattice point,
  namely $e_1+\dots+e_n$, if $k=n+1$.  This yields that
  $k(\Delta_m\times\Delta_n)$ contains exactly one interior lattice
  point if and only if $m=n=k-1$.  The claim now follows from
  \eqref{eq:Gorenstein}.
\end{proof}

The ring $K[\Delta_n]$ is isomorphic to the full polynomial ring in
$n+1$ indeterminates with coefficients in $K$.  The ring
$K[\Delta_m\times\Delta_n]$ is isomorphic to the \emph{Segre product}
of polynomial rings (with their natural gradings). Therefore,
Theorem~\ref{thm:main:var} translates into the language of Commutative
Algebra as follows: The Segre product of $K[x_0,\dots,x_m]$ and
$K[x_0,\dots,x_n]$ (with their natural gradings) is Gorenstein if and
only if $m=n$.

The point of this section is that this can be used to give an
alternate proof of our main result.

\begin{proof}[Alternate proof of Theorem~\ref{thm:main}]
  Let $P$ be a $d$-polytrope in $\TA^d$ with tropical vertices
  $v_1$, $\dots,v_n$.  We have to show that $n=d+1$.

  Now $P$ coincides with the tight span of the regular triangulation
  of $\Delta_{n-1}\times\Delta_d$ dual to the point configuration
  $(v_1,\dots,v_n)$.  In particular, this triangulation is a
  Gorenstein simplicial complex by Proposition~\ref{prop:Gorenstein}.
  So $\Delta_{n-1}\times\Delta_d$ is an ordinary polytope with a
  Gorenstein triangulation.  Since products of simplices do admit a
  regular and unimodular triangulation, for instance, the staircase
  triangulation, the result of Bruns and R\"omer,
  Theorem~\ref{thm:BrunsRoemer}, can be applied.  We derive that
  $\Delta_{n-1}\times\Delta_d$ is a Gorenstein polytope and hence
  $n=d+1$ by Theorem~\ref{thm:main:var}.
\end{proof}

It is worth to mention that Theorem~\ref{thm:BrunsRoemer} can be read
both ways.  This means that, by reversing the argument above,
Theorem~\ref{thm:main:var} is, in fact, equivalent to
Theorem~\ref{thm:main}.

\bibliographystyle{amsplain}
\bibliography{main}

\end{document}